\documentclass[12pt]{amsart}

\usepackage{amssymb}

\usepackage{enumerate}

\usepackage{pgfplots}
\usepgfplotslibrary{dateplot}
\usepackage{graphicx}
\usepackage{color}
\makeatletter
\@namedef{subjclassname@2010}{%
  \textup{2010} Mathematics Subject Classification}
\makeatother

\newtheorem{thm}{Theorem}[section]
\newtheorem{cor}[thm]{Corollary}
\newtheorem{lem}[thm]{Lemma}

\newtheorem*{conj*}{Denjoy's Conjecture}

\theoremstyle{definition}
\newtheorem{defin}[thm]{Definition}

\numberwithin{equation}{section}

\begin{document}

%%%%% To ease editing, for IMPAN journals add:

\baselineskip=17pt

%%%%%%%%%%%

%% In the running head, replace first names by initials
%% and give an abbreviation of the title.

\title[Universal commensurability augmented Teichm\"uller space and moduli space]{Universal commensurability augmented Teichm\"uller space and  moduli space}

\author[G. Hu\ H. Miyachi\ Y. Qi]{ Guangming Hu, \  Hideki Miyachi\ and Yi Qi}
\address{Guangming Hu \newline
 College of Science, Jinling Institute of  Technology,
  Nanjing, 211169, P.R. China.}
 \email{18810692738@163.com}

\address{ Hideki Miyachi \newline
School of Mathematics and Physics, College of Science and Engineering, Kanazawa University,
Kakuma-machi, Kanazawa, Ishikawa, 920-1192, Japan}
\email{miyachi@se.kanazawa-u.ac.jp}
\address{  Yi Qi
\newline
School of Mathematics and Systems Science, Beihang University,
 Beijing, 100191, P. R. China  }
\email{yiqi@buaa.edu.cn}

\thanks{This work is partially supported by NSFC Grant Number  11871085 and JSPS KAKENHI Grant Number 16K05202}

\date{}

\begin{abstract}
It is known that every  finitely unbranched  covering $\alpha:\widetilde{S}_{g(\alpha)}\rightarrow S$ of a compact Riemann surface $S$ with genus $g\geq2$ induces an isometric embedding $\Gamma_{\alpha}$ from the Teichm\"uller space $T(S)$ to the Teich\"uller space $T(\widetilde{S}_{g(\alpha)})$. Actually, it has been showed that the isometric embedding $\Gamma_{\alpha}$ can be extended isometrically to  the augmented Teichm\"{u}ller space $\widehat{T}(S)$ of $T(S)$. Using this result, we construct a directed limit $\widehat{T}_{\infty}(S)$ of augmented  Teichm\"uller spaces, where the index runs over all finitely unbranched  coverings of $S$.  Then, we show that the action of the universal commensurability modular group $Mod_{\infty}(S)$ can extend isometrically on $\widehat{T}_{\infty}(S)$. Furthermore, for any $X_{\infty}\in T_{\infty}(S)$, its orbit of the action of the universal commensurability modular group  $Mod_{\infty}(S)$  on the universal commensurability augmented Teichm\"uller space $\widehat{T}_{\infty}(S)$ is dense. Finally, we also construct a directed limit $\widehat{M}_{\infty}(S)$ of augmented moduli spaces by characteristic towers and show that the subgroup $Caut(\pi_{1}(S))$ of $Mod_{\infty}(S)$  acts on $\widehat{T}_{\infty}(S)$ to produce $\widehat{M}_{\infty}(S)$ as the quotient.
\end{abstract}



\maketitle

\section{Introduction and main results}
Let $S$ be a closed orientable surface of genus $g\geq2$. Let  $T(S)$ be the Teichm\"uller space of marked hyperbolic structures on $S$. There is a natural complete metric $d_{T (S)}$ on $T(S)$,  called the Teichm\"uller metric. The moduli space  $M(S)$ is defined  as the equivalent classes of $T(S)$ by the action of the modular group $Mod(S)$ and has  a  quotient metric $d_{M(S)}$, induced by $d_{T (S)}$.  A known compactification of $M(S)$ is called the Deligne-Mumford compactification, which is introduced in \cite{DMum} by  Deligne and  Mumford.  Abikoff introduced a partial compactification of $T(S)$, named the augmented Teichm\"{u}ller space $\widehat{T}(S)$ \cite{AW1,AW2}, whose orbit space $\widehat{M}(S)$ by the action of $Mod(S)$  is called the augmented moduli space. It is known that the augmented moduli space is homeomorphic to the Deligne-Mumford compactification \cite{Har1}.

In this paper, we consider the related issues of augmented Teichm\"{u}ller spaces and augmented moduli spaces respectively, which are divided into two parts as follows.

\subsection{Universal commensurability augmented Teichm\"uller space}

Let $\alpha:\widetilde{S}_{g(\alpha)}\rightarrow S$ be  a finitely unbranched covering  of $S$, where $\widetilde{S}_{g(\alpha)}$ is a compact Riemann surface of genus $g(\alpha)$.
It is known that the covering $\alpha$ can induce an  isometrically (endowed with Teichm\"uller metrics) holomorphic embedding $\Gamma_{\alpha} :T(S)\rightarrow T(\widetilde{S}_{g(\alpha)})$.
In \cite{BNS}, Biswas, Nag and  Sullivan studied the directed system of Teichm\"uller spaces arising from these embeddings $\Gamma_{\alpha} :T(S)\rightarrow T(\widetilde{S}_{g(\alpha)})$, as $\alpha$ runs over all finitely unbranched  coverings of $S$. This directed limit of Teichm\"uller spaces, denoted by $T_{\infty}(S)$, is called the universal commensurability Teichm\"uller space and its completion is named the Teichm\"uller space for the universal hyperbolic solenoid (See \cite{Odd,PS,Sar1,Sar2,Sar3} for details of the solenoid).  The universal commensurability Teichm\"uller space $T_{\infty}(S)$ carries a natural Weil-Petersson K\"ahler structure  from scaling the Weil-Petersson pairing on each finite dimensional stratum. Then $T_{\infty}(S)$ has a biholomorphic automorphism group, called the universal commensurability modular group $Mod_{\infty}(S)$.

In \cite{BNS},  the statement that the orbits of the action of  $Mod_{\infty}(S)$ on $T_{\infty}(S)$  are dense is actually equivalent to the following conjecture:

\noindent{\bf Ehrenpreis Conjecture }\; {\em  Let $S_{1}$ and $S_{2}$ be compact Riemann surfaces of genus at least two and $K>1$. Then there are compact Riemann surfaces $\widetilde{S}_{1}$ and $\widetilde{S}_{2}$ with  finitely holomorphic unbranched coverings $\alpha_{1}:\widetilde{S}_{1}\rightarrow S_{1}$ and $\alpha_{2}:\widetilde{S}_{2}\rightarrow S_{2}$ and a $K-$quasiconformal mapping $\widetilde{f}:\widetilde{S}_{1}\rightarrow \widetilde{S}_{2}$.}

In \cite{KM}, Kahn and  Markovic developed the notion of the good pants homology and showed that it agrees with the standard homology on closed surfaces. Then they solved the Ehrenpreis Conjecture.

  In \cite{BMN},  Biswas, Mitra and Nag  proved that the action of $Mod_{\infty}(S)$ can extend continuously on the directed limit of Thurston compactifications, and  the orbits of the action of  $Mod_{\infty}(S)$ on the boundary of  the directed limit of Thurston compactifications  are dense.

In \cite{HQ}, it has been showed that the isometric embedding $\Gamma_{\alpha}$ can extend isometrically on  the augmented Teichm\"{u}ller space of $T(S)$.
So we naturally  ask whether or not the results of Thurston compactifications can be generalized to augmented Teichm\"{u}ller spaces. Here we have the following

\begin{thm}\label{1} The action of the universal commensurability modular group  $Mod_{\infty}(S)$ can extend on the universal commensurability augmented Teichm\"uller space $\widehat{T}_{\infty}(S)$ isometrically, endowed with the Teichm\"uller metric.
\end{thm}
\noindent{\bf Remark} Since  the universal commensurability Teichm\"uller space $T_{\infty}(S)$ has a naturally complex manifold structure  from  each finite dimensional stratum, the universal commensurability modular group $Mod_{\infty}(S)$, acting on $T_{\infty}(S)$, is a biholomorphic automorphism group. However, the augmented Teichm\"uller space has no manifold structure, hence $\widehat{T}_{\infty}(S)$  has no manifold structure and the element of  $Mod_{\infty}(S)$, acting on $\widehat{T}_{\infty}(S)$, is not  holomorphic.

Combined with the Ehrenpreis conjecture, we  have the result as follows:
\begin{thm}\label{1} For any $X_{\infty}\in T_{\infty}(S)$, its orbit of the action of the universal commensurability modular group  $Mod_{\infty}(S)$  on the universal commensurability augmented Teichm\"uller space $\widehat{T}_{\infty}(S)$ is dense.
\end{thm}
\subsection{Universal commensurability augmented moduli space}

 Although the augmented Teichm\"uller space has no manifold structure, in \cite{Hubbard}, Hubbard and Koch obtained  analytic structures of  augmented moduli spaces by the method of  plumbing coordinates. In \cite{BN},  Biswas and Nag introduced the characteristic covering of $S$ defined as follows: A covering $\alpha: \widetilde{S}_{g(\alpha)}\rightarrow S$ is called characteristic if  every homeomorphism of $S$ lifts to one homeomorphism of $\widetilde{S}_{g(\alpha)}$. Therefore, it has the following
\begin{thm}\label{1}
Any characteristic covering $\alpha$, from
$\widetilde{S}_{g(\alpha)}$ to $S$, induces a holomorphic mapping $\widehat{\Phi}_{\alpha}: \widehat{M}(S)\rightarrow \widehat{M}(\widetilde{S}_{g(\alpha)})$.
\end{thm}
In \cite{BN}, Biswas and Nag  introduced a directed  sub-system corresponding to characteristic tower  coverings and showed that the subgroup $Caut(\pi_{1}(S))$ of $Mod_{\infty}(S)$  acts on $T_{\infty}(S)$ to produce the directed limit $M_{\infty}(S)$ of moduli spaces. In \cite{BMN},  Biswas, Mitra and Nag showed that $Caut(\pi_{1}(S))$ acts on the directed limit of Thurston compactifications to produce the directed limit of Thurston compactified moduli spaces as the quotient. Motivated by the study of  \cite{BMN}, we get the following

\begin{thm}\label{1} The subgroup $Caut(\pi_{1}(S))$ acts on the universal commensurability augmented Teichm\"uller space $\widehat{T}_{\infty}(S)$ to produce the directed limit $\widehat{M}_{\infty}(S)$ of augmented moduli spaces as the quotient.
\end{thm}

\section{Preliminaries}

\subsection{Augmented Teichm\"uller space}
 A marked Riemann surface modeled on $S$ is defined by one tuple $(R,f)$, where $R$ is a compact Riemann surface and $f: S\to R$ is a quasiconformal mapping. Two marked Riemann surfaces $(R_1,f_1)$ and $(R_2,f_2)$ are Teichm\"uller equivalent if there exists a conformal mapping $h:R_{1}\to R_{2}$ such that $h$ is homotopic to $f_{2}\circ f_{1}^{-1}$. The Teichm\"uller space $T(S)$  can be defined as
\[
T(S)=\{X:=[R,f]|\; \text{$(R, f)$ is a marked Riemann surface modeled on $S$}\},
\]
where $X$ is the Teichm\"{u}ller equivalent class containing $(R, f)$.

There is a natural complete metric $d_{_{T(S)}}$ on $T(S)$, called Teichm\"uller metric, which is defined as
$$
d_{T(S)}(X_{1},X_{2})=\frac{1}{2}\inf\{\log K(h)\},\;\; \text{for any } X_{1},X_{2}\in T(S),
$$
where the infimum takes over all quasiconformal mappings $h: R_{1}\to R_{2}$ homotopic to $f_{2}\circ f_{1}^{-1}$ and $K(h)$ is the maximal dilatation of $h$.

Denote by $\mathcal{S}$ the set of all homotopic classes of non-trivial simple closed curves on $S$ and  $[\gamma]$ represents the element of $\mathcal{S}$. Let $\mathcal{A}=\{\gamma_{1},\cdots, \gamma_{n}\}$ be a multicurve on $S$. A multicurve is maximal on $S$ if the multicurve has $3g-3$ components. Let $S^{\mathcal{A}}=\cup_{i=1}^{k_n}S^{i}$ be the surface  which contracts the multicurve $\mathcal{A}$ of $S$ to points and  is a homeomorphism except $\mathcal{A}$. Let $T_{S^{\mathcal{A}}}$ be the product of  Teichm\"uller spaces of the components. The augmented Teichm\"{u}ller space $\widehat{T}(S) $ of $S$ is composed by the disjoint union of strata $T_{S^{\mathcal{A}}}$. The minimal strata, which correspond to maximal multicurves, are points.

There is a natural metric $d_{\widehat{T}(S)}$ on the augmented Teich\"uller space $\widehat{T}(S)$. For any $\dot{X}_{1},\dot{X}_{2}\in\widehat{T}(S)$, if they are in the same stratum, the distance is defined as
 $$
d_{\widehat{T}(S)}(\dot{X}_{1},\dot{X}_{2})= \max_{i=1,...,k_n }d_{_{T(S^{i})}}(X_{1}^{i},X_{2}^{i}), \quad \text{for }X_{1}^{i},X_{2}^{i}\in T(S^{i}).
$$
Otherwise,
$$
d_{\widehat{T}(S)}(\dot{X}_{1},\dot{X}_{2})=+\infty.
$$
\subsection{Universal commensurability Teichm\"uller space}
Let $\alpha:\widetilde{S}_{g(\alpha)}\rightarrow S$ be  a finitely unbranched covering  of $S$.   Construct a category $\mathcal{C}(S)$ of some topological objects and morphisms as follows: the objects, denoted by $Ob(\mathcal{C}(S))$, constitute all finitely unbranched covering surfaces of $S$ equipped with a base point $(\star)$ and the morphisms, denoted by $Mor(\mathcal{C}(S))$, are based homotopic classes of pointed covering mappings $\alpha:\widetilde{S}_{g(\alpha)}\rightarrow S$. The monomorphism of fundamental groups induced by any representative of the
based homotopic class of coverings $\alpha$  is unambiguously defined.   For two morphisms $\alpha,\beta\in Mor(\mathcal{C}(S))$, we say $\alpha\prec\beta$
if there is a commuting triangle of morphisms $\beta=\alpha\circ\theta$, where $\theta$ is a finitely unbranched covering. Since we are working with surfaces with based points, the factoring morphism is uniquely determined. Therefore there exists a partially ordering relation in $Mor(\mathcal{C}(S))$ by factorizations of morphisms. It is showed in  \cite{BNS} that the morphisms $Mor(\mathcal{C}(S))$  constitute a directed set under the partially ordering relation.

It is known that a morphism $\alpha$ can induce an  isometric (endowed with Teichm\"uller metrics) embedding $\Gamma_{\alpha} :T(S)\rightarrow T(\widetilde{S}_{g(\alpha)})$. Similar to the construction of the above category, we can create a category $\mathcal{C}(T(S))$ of Teichm\"uller spaces and isometric embeddings as follows: the objects, denoted by $Ob(\mathcal{C}(T(S)))$, constitute all Teichm\"uller spaces induced by finitely unbranched coverings of $S$ and the morphisms, denoted by $Mor(\mathcal{C}(T(S)))$, are all isometric embeddings $\Gamma_{\alpha} :T(S)\rightarrow T(\widetilde{S}_{g(\alpha)})$ for $\alpha\in\mathcal{C}(S)$. For each $\alpha\prec\beta$, we have the corresponding isometric embedding $\Gamma_{\theta}$, satisfied with $\Gamma_{\beta}=\Gamma_{\theta}\circ \Gamma_{\alpha}$, where $\beta=\alpha\circ\theta$.
There exists a partially ordering relation in $\mathcal{C}(T(S))$ by factorizations of isometric embeddings. Then the morphisms $Mor(\mathcal{C}(T(S)))$ produce a natural directed system induced by $Mor(\mathcal{C}(S))$.

 From this directed system,  the directed limit Teichm\"uller space over $S$ is defined as $
T_{\infty}(S):=\text{dir.}\lim.T(\widetilde{S}_{g(\alpha)})$ for $\alpha\in Mor(\mathcal{C}(S))$ which is named the universal commensurability Teichm\"uller space. It is an inductive limit of finite dimensional spaces \cite{Sha}.

The Teichm\"uller metric $d_{T_{\infty}(S)}$ on $T_{\infty}(S)$ is defined as the Teichm\"uller metric $d_{T(\widetilde{S}_{g(\alpha)})}$ on every stratification $T(\widetilde{S}_{g(\alpha)})$ for $\alpha\in Mor(\mathcal{C}(S))$.

\subsection{Universal commensurability modular group}
 It is showed in \cite{BNS} that every morphism $\alpha:\widetilde{S}_{g(\alpha)}\rightarrow S$ can  induce  a natural Teichm\"uller metric preserving homeomorphism $\Gamma_{\infty}(\alpha):T_{\infty}(\widetilde{S}_{g(\alpha)})\rightarrow T_{\infty}(S)$ because the directed set  $Mor(\mathcal{C}(\widetilde{S}_{g(\alpha)}))$ is cofinal in  $Mor(\mathcal{C}(S))$.

 For a given pair of morphisms
$$
\alpha: \widetilde{S}\rightarrow S \quad \text{and}\quad \beta:\widetilde{S}\rightarrow S,
$$
we have an isometric automorphism $\Gamma_{\infty}(\beta)\circ\Gamma_{\infty}^{-1}(\alpha)$ on $T_{\infty}(S)$.
More  generally,  we are given a cycle of morphisms starting and ending at $S$ as follows:
\begin{align*}
&\widetilde{S}_{k}\quad-\quad \widetilde{S}_{k+1}\\
&\:\big|\qquad\qquad\;\big|\\
&\widetilde{S}_{k-1}\quad\quad \widetilde{S}_{k+2}\\
&\:\big|\qquad\qquad\;\big|\\
&\,\vdots\qquad\qquad\;\vdots\\
&\widetilde{S}_{1}\qquad\quad\: \widetilde{S}_{n}\\
&\:\big|\qquad\qquad\;\big|\\
&S\,\quad=\quad\,S
\end{align*}
where $\widetilde{S}_{i},S $ are all objects of the category  $\mathcal{C}(S)$  and all horizontal and vertical lines represent morphisms of  $\mathcal{C}(S)$. Since $\Gamma_{\infty}(\alpha)$ is invertible, the horizontal and the vertical lines in the above cycle of morphisms are allowed  in any direction. Thus we can define an isometric automorphism   $\varphi:T_{\infty}(S)\rightarrow T_{\infty}(S)$ which is a composition around the undirect cycle of morphisms starting from $S$ and returning to $S$. The group $Mod_{\infty}(S)$ consists of  all thus automorphism, called the universal commensurability modular group, acting on $T_{\infty}(S)$ isometrically.

\subsection{ Complex structure of augmented moduli space }

The modular group $Mod(S)$ is defined as the group of homotopic equivalent classes of orientation-preserving homeomorphisms of $S$.
 Two orientation preserving homeomorphisms $\phi,\varphi:S\rightarrow S$ are equivalent if $\phi\circ\varphi^{-1}$ is isotopic to the identity map on $S$. Since $Mod(S)$ acts  on $(T(S),d_{T(S)})$ discretely and isometrically,  the moduli space  $M(S)$ is defined  as the equivalent classes of $T(S)$ by the action of $Mod(S)$. Let $\mathcal{X}:=\{R,f\}$ be represented as the element of $M(S)$. The orbit space $\widehat{M}(S)=\widehat{Teich}(S)/Mod(S)$ is called  the augmented moduli space  and  is compact with the quotient topology \cite{AW2}. Denote by $\dot{\mathcal{X}}$ the element of $\widehat{M}(S)$.

In \cite{Hubbard}, Hubbard and Koch introduced some subgroups of the modular group $Mod(S)$ as follows:
\begin{defin}\label{1} Denote by $\mathcal{A}$  a multicurve on $S$. Let $S/\mathcal{A}$ be the topological surface obtained from collapsing the elements $\mathcal{A}$ to points. Then we can define the following groups:

$\bullet$  $Mod(S,\mathcal{A})$ is the subgroup of $Mod(S)$ consisting of those elements which have representative homeomorphisms $\varphi: S\rightarrow S$, such that for all $\gamma\in \mathcal{A}$, $\varphi: [\gamma]\rightarrow [\gamma]$, and $\varphi$ fixes each component of $S-[\mathcal{A}]$, where $[\mathcal{A}]$ is the set of all homotopic classes of $\mathcal{A}$  on $S$;

$\bullet$ $Mod(S/\mathcal{A})$ is the subgroup of isotopy classes of homeomorphisms $S/\mathcal{A}\rightarrow S/\mathcal{A}$ that fix the image of each $\gamma\in\mathcal{A}$ in $S/\mathcal{A}$ and map each component of $S-\mathcal{A}$ to itself;

$\bullet$ $\Delta_{\mathcal{A}}(S)$ is the subgroup of $Mod(S)$ generated by Dehn twists around the Teichm\"uller space.
\end{defin}

Consider the space
$$
U_{\mathcal{A}}(S):=\bigcup_{\mathcal{A}'\subseteq\mathcal{A}}T_{S^{\mathcal{A}'}}\subseteq\widehat{T}(S),
$$
and the subgroup $\Delta_{\mathcal{A}}(S)$ of  $Mod(S)$ acts on $U_{\mathcal{A}}(S)$.  The space
$$
\mathcal{Q}_{\mathcal{A}}(S):=U_{\mathcal{A}}(S)/\Delta_{\mathcal{A}}(S)
$$
 is  the quotient topology inherited from $\widehat{T}(S)$. The stratum of $\mathcal{Q}_{\mathcal{A}}(S)$ is defined as $T_{S^{\mathcal{A}'}}/\Delta_{\mathcal{A}}(S)$, denoted by $\mathcal{Q}^{\mathcal{A}'}_{\mathcal{A}}(S)$.

If  complete $\mathcal{A}$ to a maximal multicurve $\mathcal{A}_{\max}$, the Fenchel-Nielsen coordinate of $\mathcal{Q}_{\mathcal{A}}(S)$  represents $(\mathbb{R}_{+}\times\mathbb{R})^{\mathcal{A}_{\max}-\mathcal{A}}\times\mathbb{C}^{\mathcal{A}}$ \cite{Hubbard}.

For $\dot{X}_{0}\in\mathcal{Q}^{\mathcal{A}}_{\mathcal{A}}(S)$, there exists an open neighborhood $T$ of $\dot{X}_{0}$ in $\mathcal{Q}^{\mathcal{A}}_{\mathcal{A}}(S)$. Let $\mathcal{P}_{\mathcal{A}}(S)=T\times \mathbb{D}^{\mathcal{A}}$ be a complex manifold and it is the union of strata
$$
\mathcal{P}_{\mathcal{A}}(S)=\bigcup_{\mathcal{A}'\subset\mathcal{A}}\mathcal{P}^{\mathcal{A}'}_{\mathcal{A}}(S),
$$
where
$$
\mathcal{P}^{\mathcal{A}'}_{\mathcal{A}}(S)=\{(\dot{X},\mathbf{z})\in T\times\mathbb{D}^{\mathcal{A}}|z_{\gamma}=0\Leftrightarrow \gamma\in\mathcal{A}'\}.
$$
The Fenchel-Nielsen coordinate of $\mathcal{P}_{\mathcal{A}}(S)$ is defined as
$$
(l_{\gamma},\tau_{\gamma}),\gamma\in\mathcal{A}_{\max}-\mathcal{A};\quad l_{\gamma}e^{2\pi i\tau_{\gamma}/ l_{\gamma}},\gamma\in\mathcal{A}
$$
and it can define a mapping $\Psi:\mathcal{P}_{\mathcal{A}}(S)\rightarrow\mathcal{Q}_{\mathcal{A}}(S)$.  Hubbard and Koch proved the following result:
\begin{lem}\label{1}  For any $(\dot{X},\mathbf{t})\in\mathcal{P}_{\mathcal{A}}(S)$ with $\|\mathbf{t}\|$ sufficiently small, there exists  a neighborhood $V'$ such that $V:=\Psi(V')$ is an open set in $\mathcal{Q}_{\mathcal{A}}(S)$ and $\Psi:V'\rightarrow V$ is a homeomorphism.
\end{lem}
From Lemma $2.2$, they constructed the complex structure of $\mathcal{Q}_{\mathcal{A}}(S)$ as follows:
 	\vskip 4pt

\noindent{\bf Theorem A.}\; {\em For every multicurve $\mathcal{A}$ on $S$, there exists a complex manifold structure on $\mathcal{Q}_{\mathcal{A}}(S)$. }

Then the union $\bigcup_{\mathcal{A}}\mathcal{Q}_{\mathcal{A}}(S)$ of  $\mathcal{Q}_{\mathcal{A}}(S)$ over all multicurves $\mathcal{A}$ is a complex manifold. Since
$$\widehat{M}(S)=\bigg(\bigcup_{\mathcal{A}}\mathcal{Q}_{\mathcal{A}}(S)\bigg)/Mod(S),$$ they obtained the complex structure on $\widehat{M}(S)$ as follows:

\noindent{\bf Theorem B.}\; {\em The union of the images of $\pi_{\mathcal{A}}:\mathcal{Q}_{\mathcal{A}}(S)\rightarrow\widehat{M}(S)$ over all multicurves $\mathcal{A}$ covering  $\widehat{M}(S)$ give $\widehat{M}(S)$ the structure of an analytic orbifold. }
\subsection{ Universal commensurability  moduli space}
 A morphism $\alpha: \widetilde{S}_{g(\alpha)}\rightarrow S$ is called characteristic if the fundamental group $\pi_{1}(\widetilde{S}_{g(\alpha)})$ is a characteristic subgroup of the fundamental group $\pi_{1}(S)$. In other words, the subgroup $\pi_{1}(\widetilde{S}_{g(\alpha)})\subseteq \pi_{1}(S)$ must be invariant by every  automorphism of $\pi_{1}(S)$. This can  yield a monomorphism: $L_{\alpha}:Aut(\pi_{1}(S))\rightarrow Aut(\pi_{1}(\widetilde{S}_{g(\alpha)}))$. The topological characterization of a characteristic cover is that every homeomorphism of $S$ lifts to one homeomorphism of $\widetilde{S}_{g(\alpha)}$, and the homomorphism $L_{\alpha}$ corresponds to this lifting process. The characteristic subgroups of finite index form a cofinal family among all subgroups of finite index in $\pi_{1}(S)$ \cite{BN}.

 Consider the characteristic tower $Mor(\mathcal{C}^{ch}(S))$ over $S$ consisting of only characteristic morphisms. For $\alpha,\beta\in Mor(\mathcal{C}^{ch}(S))$, we say $\alpha\prec_{ch}\beta$ if and only if $\beta=\alpha\circ \theta$ which $\theta$ is also a characteristic morphism. The characteristic tower $Mor(\mathcal{C}^{ch}(S))$ is  a  directed set  under the partial ordering given by factorization of characteristic morphisms.  It is known that the characteristic tower $Mor(\mathcal{C}^{ch}(S))$ is the cofinal subset of $Mor(\mathcal{C}(S))$ \cite{BN}.  Any characteristic morphism $\alpha$, from
$\widetilde{S}_{g(\alpha)}$ to $S$, induces a morphism $\Phi_{\alpha}: M(S)\rightarrow M(\widetilde{S}_{g(\alpha)})$ which is an algebraic morphism between these normal quasi-projective varieties.
Similarly, we have a directed limit $M_{\infty}(S)$ of moduli spaces over $S$ as follows:
$$
 M_{\infty}(S):=\text{dir.}\lim.M(\widetilde{S}_{g(\alpha)}), \quad \alpha\in Mor(\mathcal{C}^{ch}(S)),
$$
which is called the universal commensurability  moduli space.

Using the monomorphisms $L_{\alpha}:Aut(\pi_{1}(S))\rightarrow Aut(\pi_{1}(\widetilde{S}_{g(\alpha)}))$ , we can define a directed limit of automorphism groups over $S$ as follows:
$$
Caut(\pi_{1}(S))=\text{dir.}\lim.\text{Aut}(\pi_{1}(\widetilde{S}_{g(\alpha)}), \quad \alpha\in Mor(\mathcal{C}^{ch}(S)).
$$

\section{Universal commensurability augmented Teichm\"uller space and universal commensurability modular group}
In this section, we give the definition of the universal commensurability augmented Teich\"uller space over $S$. Then we consider the action of the universal commensurability modular group on the universal commensurability augmented Teich\"uller space.

The following lemma \cite{HQ} play an important role in defining the universal commensurability augmented Teich\"uller space.

 \begin{lem}\label{1}
 Let $\alpha:\widetilde{S}_{g(\alpha)}\rightarrow S$ be a finitely unbranched holomorphic covering of a compact Riemann surface $S$ with genus $g\geq2$. Then  the isometric embedding $\Gamma_{\alpha}$ can   extend  to  $\widehat{\Gamma}_{\alpha}:\widehat{T}(S)\rightarrow\widehat{T}(\widetilde{S}_{g(\alpha)})$ isometrically.
 \end{lem}
 Similar to the universal commensurability Teichm\"uller space, we can create a category $\mathcal{C}(\widehat{T}(S))$ of augmented Teichm\"uller spaces and isometric embeddings as follows: the objects, denoted by $Ob(\mathcal{C}(\widehat{T}(S)))$, constitute all augmented Teichm\"uller spaces induced by finitely unbranched coverings of $S$ and the morphisms, denoted by $Mor(\mathcal{C}(\widehat{T}(S)))$, are all isometric embeddings $\widehat{\Gamma}_{\alpha} :\widehat{T}(S)\rightarrow \widehat{T}(\widetilde{S})$. For each $\alpha\prec\beta$, we have the corresponding isometric embedding $\widehat{\Gamma}_{\theta}$, satisfied with $\widehat{\Gamma}_{\beta}=\widehat{\Gamma}_{\theta}\circ \widehat{\Gamma}_{\alpha}$, where $\beta=\alpha\circ\theta$.
There exists a partially ordering relation in $Mor(\mathcal{C}(\widehat{T}(S)))$ by factorizations of isometric embeddings. Then the morphisms $Mor(\mathcal{C}(S))$  constitute a directed set under the partially ordering relation. Then the morphisms $Mor(\mathcal{C}(\widehat{T}(S)))$ produce a natural directed system induced by $Mor(\mathcal{C}(S))$.
   \begin{defin}\label{1} For the directed system $Mor(\mathcal{C}(\widehat{T}(S)))$,  the directed limit of augmented Teichm\"uller spaces over $S$ is defined as
 $$
  \widehat{T}_{\infty}(S):=\text{dir.}\lim.\widehat{T}(\widetilde{S}_{g(\alpha)}), \quad \alpha\in Mor(\mathcal{C}(S))
  $$
  which is called the universal commensurability augmented Teichm\"uller space.
\end{defin}

\begin{defin}\label{1}
The Teichm\"uller metric $d_{\widehat{T}_{\infty}(S)}$ on $\widehat{T}_{\infty}(S)$ is defined as the Teichm\"uller metric $d_{\widehat{T}(\widetilde{S}_{g(\alpha)})}$ on every stratification $\widehat{T}(\widetilde{S}_{g(\alpha)})$ for $\alpha\in Mor(\mathcal{C}(S))$.
\end{defin}

Next, we show that the action of the universal commensurability modular group  $Mod_{\infty}(S)$ is  isometric, endowed with the Teichm\"uller metric, on $\widehat{T}_{\infty}(S)$.

	\vskip 4pt

\noindent{\bf Proof of Theorem 1.1}\; {\em Suppose that  $\alpha:\widetilde{S}_{g(\alpha)}\rightarrow S $ is a morphism, then it can  induce  a natural Teichm\"uller metric preserving homeomorphism $\Gamma_{\infty}(\alpha):T_{\infty}(\widetilde{S}_{g(\alpha)})\rightarrow T_{\infty}(S)$. According to the definition of $T_{\infty}(\widetilde{S}_{g(\alpha)})$, it is easy to know any stratification $T(\widetilde{S}_{g(\delta)})$, $\delta\in Mor(\mathcal{C}(\widetilde{S}_{g(\alpha)}))$, of
$T_{\infty}(\widetilde{S}_{g(\alpha)})$ can be embedded in some stratification $T(S_{g(\eta)})$, $\eta\in Mor(\mathcal{C}(S))$, of $T_{\infty}(S)$ isometrically. By Lemma $3.1$, the isomeric embedding  can   extend  to  the corresponding augmented Teichm\"uller space isometrically, then we have  a natural Teichm\"uller metric preserving embedding $\widehat{\Gamma}_{\infty}(\alpha):\widehat{T}_{\infty}(\widetilde{S}_{g(\alpha)})\rightarrow \widehat{T}_{\infty}(S)$ by the definition of  the universal commensurability augmented Teichm\"uller space.

Since $\Gamma_{\infty}(\alpha):T_{\infty}(\widetilde{S}_{g(\alpha)})\rightarrow T_{\infty}(S)$ is a homeomorphism, we have the invertible mapping $\Gamma^{-1}_{\infty}(\alpha):T_{\infty}(S)\rightarrow T_{\infty}(\widetilde{S}_{g(\alpha)})$.  Then  any stratification $T(S_{g(\eta)})$, $\eta\in Mor(\mathcal{C}(S))$, of $T_{\infty}(S)$  also can be embedded in some stratification  $T(\widetilde{S}_{g(\delta)})$, $\delta\in Mor(\mathcal{C}(\widetilde{S}))$, of
$T_{\infty}(\widetilde{S})$  isometrically.  By Lemma $3.1$ again, the isomeric embedding  can   extend  to  the corresponding augmented Teichm\"uller space isometrically, hence $\widehat{\Gamma}_{\infty}(\alpha):\widehat{T}_{\infty}(\widetilde{S}_{g(\alpha)})\rightarrow \widehat{T}_{\infty}(S)$ is surjective.
Here we  show that $\alpha$ can induce a natural Teichm\"uller metric preserving homeomorphism $\widehat{\Gamma}_{\infty}(\alpha):\widehat{T}_{\infty}(\widetilde{S}_{g(\alpha)})\rightarrow \widehat{T}_{\infty}(S)$ .

For a given pair of morphisms
$$
\alpha: \widetilde{S}\rightarrow S \quad \text{and}\quad \beta:\widetilde{S}\rightarrow S,
$$
there exists  a natural  isometric automorphism $\widehat{\Gamma}_{\infty}(\beta)\circ\widehat{\Gamma}_{\infty}^{-1}(\alpha)$ on $\widehat{T}_{\infty}(S)$.
Since any element $\varphi\in Mod_{\infty}(S)$ can be induced by  the following cycle of morphisms starting and ending at $S$
\begin{align*}
&\widetilde{S}_{k}\quad-\quad \widetilde{S}_{k+1}\\
&\:\big|\qquad\qquad\;\big|\\
&\widetilde{S}_{k-1}\quad\quad \widetilde{S}_{k+2}\\
&\:\big|\qquad\qquad\;\big|\\
&\,\vdots\qquad\qquad\;\vdots\\
&\widetilde{S}_{1}\qquad\quad\: \widetilde{S}_{n}\\
&\:\big|\qquad\qquad\;\big|\\
&S\,\quad=\quad\,S,
\end{align*}
  $\varphi$ can  extend on  the universal commensurability augmented Teichm\"uller space $\widehat{T}_{\infty}(S)$ isometrically by the above method. Therefore the action of the group $Mod_{\infty}(S)$ can extend on  $\widehat{T}_{\infty}(S)$ isometrically.}

\vskip 4pt

\noindent{\bf Proof of  Theorem 1.2 }\; {\em Since the  Ehrenpreis Conjecture is affirmatively solved, we have that the orbits of the action of  $Mod_{\infty}(S)$ on $T_{\infty}(S)$  are dense. Namely, for any $X_{\infty}\in T_{\infty}(S)$, its orbit of the action of the universal commensurability modular group  $Mod_{\infty}(S)$  on $T_{\infty}(S)$ is dense. For any $\dot{X}_{\infty}\in \widehat{T}_{\infty}(S)$, there exists a sequence $\{X^{n}_{\infty}\}_{n=1}^{\infty}\subseteq T_{\infty}(S)$ converging to  $\dot{X}_{\infty}$. Since the orbits of the action of  $Mod_{\infty}(S)$ on $T_{\infty}(S)$  are dense, for any $X^{n}_{\infty}$, there is a sequence $\{X^{n,m}_{\infty}\}_{m=1}^{\infty}$ in the orbit of $X_{\infty}$ converging to it. Then we can choose a sequence $\{X^{n,n}_{\infty}\}_{n=1}^{\infty}$ converging to $\dot{X}_{\infty}$ by the Cantor diagonal method. Therefore,  the orbit of  $X_{\infty}$ on the universal commensurability augmented Teichm\"uller space $\widehat{T}_{\infty}(S)$ is dense.
}

\section{ Universal commensurability  augmented moduli space}
In this section, we show that any characteristic covering $\alpha$, from $\widetilde{S}_{g(\alpha)}$ to $S$, induces an analytic mapping $\widehat{\Phi}_{\alpha}: \widehat{M}(S)\rightarrow \widehat{M}(\widetilde{S}_{g(\alpha)})$.
Then we construct the universal commensurability augmented moduli space and show that  the subgroup $Caut(\pi_{1}(S))$ acts on the universal commensurability augmented Teichm\"uller space $\widehat{T}_{\infty}(S)$ to produce the directed limit $\widehat{M}_{\infty}(S)$ of augmented moduli spaces as the quotient.

 Let $\alpha:\widetilde{S}_{g(\alpha)}\rightarrow S$ be a  characteristic covering of a compact Riemann surface $S$ with genus $g\geq2$. For any  multicurve $\mathcal{A}$ on $S$, let  $\widetilde{\mathcal{A}}_{g(\alpha)}:=\alpha^{-1}(\mathcal{A})$ be the preimage  on $\widetilde{S}_{g(\alpha)}$ of $\mathcal{A}$.

\begin{lem}\label{1}  Let $\alpha:\widetilde{S}_{g(\alpha)}\rightarrow S$ be a  characteristic covering of a compact Riemann surface $S$ with genus $g\geq2$. For every multicurve $\mathcal{A}$ on $S$, then there exists a holomorphic mapping $\mathcal{E}_{\alpha}:\mathcal{Q}_{\mathcal{A}}(S)\rightarrow\mathcal{Q}_{\widetilde{\mathcal{A}}_{g(\alpha)}}(\widetilde{S}_{g(\alpha)})$.
\end{lem}

\begin{proof}  From Lemma $3.1$, there is an isometric embedding $\widehat{\Gamma}_{\alpha}:U_{\mathcal{A}}(S)\rightarrow U_{\widetilde{\mathcal{A}}_{g(\alpha)}}(\widetilde{S}_{g(\alpha)})$. Since the covering mapping $\alpha $ can induce a monomorphism from the subgroup $\Delta_{\mathcal{A}}(S)$ of $Mod(S)$ to the subgroup $\Delta_{\widetilde{\mathcal{A}}_{g(\alpha)}}(\widetilde{S}_{g(\alpha)})$ of $Mod(\widetilde{S}_{g(\alpha)})$, the isometric embedding $\widehat{\Gamma}_{\alpha}:U_{\mathcal{A}}(S)\rightarrow U_{\widetilde{\mathcal{A}}_{g(\alpha)}}(\widetilde{S}_{g(\alpha)})$ can induce a continuous mapping $\mathcal{E}_{\alpha}:\mathcal{Q}_{\mathcal{A}}(S)\rightarrow \mathcal{Q}_{\widetilde{\mathcal{A}}_{g(\alpha)}}(\widetilde{S}_{g(\alpha)})$.

For any  point $p$ of $\mathcal{Q}^{\mathcal{A}}_{\mathcal{A}}(S)$, there exists a point $\widetilde{p}:=\mathcal{E}_{\alpha}(p)$ in $\mathcal{Q}^{\widetilde{\mathcal{A}}_{g(\alpha)}}_{\widetilde{\mathcal{A}}_{g(\alpha)}}(\widetilde{S}_{g(\alpha)})$. In addition, using Lemma $2.2$,  there exists  a  neighborhood $\widetilde{V}_{1\ast}\subset\mathcal{P}_{\widetilde{\mathcal{A}}_{g(\alpha)}}(\widetilde{S}_{g(\alpha)})$  such that $\widetilde{V}_{1}=\widetilde{\Psi}(\widetilde{V}_{1\ast})$ is an open set of $\widetilde{p}$ in $\mathcal{Q}_{\widetilde{\mathcal{A}}_{g(\alpha)}}(\widetilde{S}_{g(\alpha)})$ and $\widetilde{\Psi}:\widetilde{V}_{1\ast}\rightarrow \widetilde{V}_{1}$ is a homeomorphism. Since $\mathcal{E}_{\alpha}$ is a continuous mapping,  $V_{1}=\mathcal{E}_{\alpha}^{-1}(\widetilde{V}_{1})$ is an open set  in $ \mathcal{Q}_{\mathcal{A}}(S)$. Using Lemma $2.2$ again, there exists  a  neighborhood $V_{2\ast}\subset\mathcal{P}_{\mathcal{A}}(S)$ such that $V_{2}=\Psi(V_{2\ast})$ is an open set in $\mathcal{Q}_{\mathcal{A}}(S)$ and $\Psi:V_{2\ast}\rightarrow V_{2}$ is a homeomorphism. Then it has a continuous mapping $\widetilde{\Psi}^{-1}\circ\mathcal{E}_{\alpha}\circ\Psi: \Psi^{-1}(V_{1}\cap V_{2})\rightarrow \widetilde{\Psi}^{-1}\circ\mathcal{E}_{\alpha}(V_{1}\cap V_{2})$, where $\Psi^{-1}(V_{1}\cap V_{2})\subset\mathcal{P}_{\mathcal{A}}(S)=T\times \mathbb{D}^{\mathcal{A}}$ and $\widetilde{\Psi}^{-1}\circ\mathcal{E}_{\alpha}(V_{1}\cap V_{2})\subset\mathcal{P}_{\widetilde{\mathcal{A}}_{g(\alpha)}}(\widetilde{S}_{g(\alpha)})=\widetilde{T}_{g(\alpha)}\times \mathbb{D}^{\widetilde{\mathcal{A}}_{g(\alpha)}}$ are  complex sub-manifolds. We know that  $T\subset T_{S^{\mathcal{A}}}$ and $\widetilde{T}_{g(\alpha)}\subset T_{\widetilde{S}^{\widetilde{\mathcal{A}}_{g(\alpha)}}}$ are two products of  Teichm\"uller spaces of the components. It is showed in \cite{KI} that there exists a  holomorphic embedding  from every factor of  $T_{S^{\mathcal{A}}} $ to the corresponding factor of $T_{\widetilde{S}^{\widetilde{\mathcal{A}}_{g(\alpha)}}}$ by  $\alpha:\widetilde{S}_{g(\alpha)}\rightarrow S$. For $\gamma\in\mathcal{A}$,
the factor of $\gamma$ is $z_{\gamma}= l_{\gamma}e^{2\pi i\tau_{\gamma}/ l_{\gamma}}\in\mathbb{D}$, then the factor of one component $\widetilde{\gamma}$ of $\alpha^{-1}(\gamma)$ can be represented as $z_{\widetilde{\gamma}}= \widetilde{l}_{\gamma}e^{2\pi i\widetilde{\tau}_{\gamma}/ \widetilde{l}_{\gamma}}=ml_{\gamma}e^{2\pi im\tau_{\gamma}/ (ml_{\gamma})}=mz_{\gamma}$ in a neighborhood of the original point, where $m$ is a positive integer. Therefore, $\widetilde{\Psi}^{-1}\circ\mathcal{E}_{\alpha}\circ\Psi: \Psi^{-1}(V_{1}\cap V_{2})\rightarrow \widetilde{\Psi}^{-1}\circ\mathcal{E}_{\alpha}(V_{1}\cap V_{2})$ is a holomorphic mapping.

For any point $p$ of $\mathcal{Q}_{\mathcal{A}}(S)-\mathcal{Q}^{\mathcal{A}}_{\mathcal{A}}(S)$, it belongs to some stratum $\mathcal{Q}^{\mathcal{A}'}_{\mathcal{A}}(S)$. According to Theorem $A,$ there is a holomorphic map $\mathcal{P}^{\mathcal{A}'}_{\mathcal{A}}:\mathcal{Q}_{\mathcal{A}'}(S)\rightarrow\mathcal{Q}_{\mathcal{A}}(S)$, which consists of quotienting by $\Delta_{\mathcal{A}-\mathcal{A}'}(S)$.   Since the curves of $\mathcal{A}-\mathcal{A}'$ are not collapsed at $p$, the curves of $\widetilde{\mathcal{A}}_{g(\alpha)}-\widetilde{\mathcal{A}}'_{g(\alpha)}$ are not collapsed at $\widetilde{p}:=\mathcal{E}_{\alpha}(p)$.  Choose as a local chart at $p$ a section of $\mathcal{P}^{\mathcal{A}'}_{\mathcal{A}}$ over a neighborhood of $p$ contained in $\mathcal{Q}_{\mathcal{A}}(S)$. Let $p'$  be a preimage of $p$ in the section of  $\mathcal{P}^{\mathcal{A}'}_{\mathcal{A}}$ . Similar to  the above analysis, there exist two neighborhoods $V_{1}$ and $V_{2}$ of $p$ in $\mathcal{Q}_{\mathcal{A}}(S)$ such that $(\widetilde{\Psi}')^{-1}\circ\big(\mathcal{P}^{\widetilde{\mathcal{A}}'_{g(\alpha)}}_{\widetilde{\mathcal{A}}_{g(\alpha)}}\big)^{-1}\circ\mathcal{E}_{\alpha}
\circ\mathcal{P}^{\mathcal{A}'}_{\mathcal{A}}\circ\Psi': (\Psi')^{-1}\circ\big(\mathcal{P}^{\mathcal{A}'}_{\mathcal{A}}\big)^{-1}(V_{1}\cap V_{2})\rightarrow (\widetilde{\Psi}')^{-1}\circ\big(\mathcal{P}^{\widetilde{\mathcal{A}}'_{g(\alpha)}}_{\widetilde{\mathcal{A}}_{g(\alpha)}}\big)^{-1}
\circ\mathcal{E}_{\alpha}(V_{1}\cap V_{2})$ is a holomorphic mapping, where $\big(\mathcal{P}^{\widetilde{\mathcal{A}}'_{g(\alpha)}}_{\widetilde{\mathcal{A}}_{g(\alpha)}}\big)^{-1}\circ\mathcal{E}_{\alpha}
\circ\mathcal{P}^{\mathcal{A}'}_{\mathcal{A}}:\mathcal{Q}_{\mathcal{A}'}(S)\rightarrow\mathcal{Q}_{\widetilde{\mathcal{A}}'_{g(\alpha)}}(\widetilde{S}_{g(\alpha)})$ is a continuous mapping.

Therefore, we obtain a holomorphic mapping $\mathcal{E}_{\alpha}:\mathcal{Q}_{\mathcal{A}}(S)\rightarrow\mathcal{Q}_{\widetilde{\mathcal{A}}_{g(\alpha)}}(\widetilde{S}_{g(\alpha)})$.

\end{proof}

\noindent{\bf Proof of Theorem 1.3}\; {\em Since  $\alpha:\widetilde{S}_{g(\alpha)}\rightarrow S$ is a  characteristic covering,  then there exists a monomorphism: $L_{\alpha}: Mod(S)\rightarrow Mod(\widetilde{S}_{g(\alpha)})$. From Lemma $4.1$, there is a holomorphic mapping
$$
\mathcal{E}_{\alpha}:\bigcup_{\mathcal{A}}\mathcal{Q}_{\mathcal{A}}(S)\rightarrow\bigcup_{\widetilde{\mathcal{A}}_{g(\alpha)}}
\mathcal{Q}_{\widetilde{\mathcal{A}}_{g(\alpha)}}(\widetilde{S}_{g(\alpha)}),
$$
then
$$
\widehat{\Phi}_{\alpha}:\bigg(\bigcup_{\mathcal{A}}\mathcal{Q}_{\mathcal{A}}(S)\bigg)/Mod(S)
\rightarrow\bigg(\bigcup_{\widetilde{\mathcal{A}}_{g(\alpha)}}\mathcal{Q}_{\widetilde{\mathcal{A}}_{g(\alpha)}}(\widetilde{S}_{g(\alpha)})\bigg)
/Mod(\widetilde{S}_{g(\alpha)})
$$ is a holomorphic mapping. This theorem is proved completely.}

 \begin{defin}\label{1} From the directed set $Mor(\mathcal{C}^{ch}(S))$,  the directed limit of augmented moduli spaces over $S$ is defined as follows:
 $$
 \widehat{M}_{\infty}(S):=\text{dir.}\lim.\widehat{M}(\widetilde{S}_{g(\alpha)}), \quad \alpha\in Mor(\mathcal{C}^{ch}(S)),
$$
which is called the universal commensurability augmented moduli space.
\end{defin}

\begin{cor}\label{1}The  universal commensurability augmented moduli space $\widehat{M}_{\infty}(S)$ has a complex structure.
\end{cor}
Next, we show that the subgroup $Caut(\pi_{1}(S))$ acts on $\widehat{T}_{\infty}(S)$ to produce $\widehat{M}_{\infty}(S)$ as the quotient.

	\vskip 4pt

\noindent{\bf Proof of Theorem 1.4}\; {\em The directed system $Mor(\mathcal{C}^{ch}(\widehat{T}(S)))$ of augmented Teichm\"uller spaces is  the cofinal sub-system of $Mor(\mathcal{C}(\widehat{T}(S)))$ and set $\widehat{T}^{ch}_{\infty}(S) $ the corresponding directed limit space. The inclusion of $\mathcal{C}^{ch}(S)$ in $\mathcal{C}(S)$ induces a natural homeomorphism of $\widehat{T}^{ch}_{\infty}(S) $ onto $\widehat{T}_{\infty}(S) $. Clearly, it follows from the definition of the subgroup $Caut(\pi_{1}(S))$  that  $Caut(\pi_{1}(S))$ acts on $\widehat{T}^{ch}_{\infty}(S) $ to
produce $\widehat{M}_{\infty}(S)$ as the quotient. Therefore, identifying $\widehat{T}^{ch}_{\infty}(S) $ with $\widehat{T}_{\infty}(S) $ by the
above homeomorphism,  we obtain the result.}

\section{Acknowledgements}
 We  thank the referees for many important and useful comments!


\scriptsize




\begin{thebibliography}{HD}

\normalsize
\baselineskip=17pt



\bibitem{AW1}W. Abikoff,  {\it Augmented Teichm\"uller spaces}, Bull. Amer. Math. Soc., 1976, 82: 333-334.
\bibitem{AW2}W. Abikoff,  {\it Degenerating Families of Riemann Surfaces}, Ann. of Math.,   1977, 105: 29-44.

   \bibitem{BMN} I. Biswas, M. Mitra, and S. Nag, {\it Thurston boundary of the Teichm\"uller spaces and the commensurability modular group}, Conform. Geom. Dyn., (3) 1999,  50-66.
  \bibitem{BN}  I. Biswas and S. Nag,{\it Weil-Petersson geometry and determinant bundles over inductive
limits of moduli spaces},  Contemporary Math.,  (211) 1997, 51-80.

\bibitem{BNS} I. Biswas, S. Nag and D. Sullivan, {\it Determinant bundles, Quillen metrics and Mumford
isomorphisms over the Universal Commensurability Teichm\"uller Space}, Acta Math., (176) 1996, 145-169.
 \bibitem{DMum} P. Deligne and D. Mumford, {\it The irreducibility of the space of curves of given genus}, Inst. Hautes Etudes Sci. Publ. Math., (36) 1969, 75-109.
 \bibitem{Har1}  W. Harvey, {\it Chabauty spaces of discrete groups,} In  Discontinuous groups
and Riemann surfaces. Vol. 79,  1974, 239-246.

\bibitem{HQ}G. Hu and Y Qi, {\it The Isometric Embedding of the Augmented Teichm\"uller Space of a Riemann Surface Into the Augmented Teichm\"uller Space of Its Covering Surface}, Kodai Math. J.,  (42) 2019, 376-392.
\bibitem{Hubbard}J. Hubbard and S. Koch,{\it An analytic construction of the Deligne-Mumford compactification of the moduli space of curves,} J. Differential
Geom., (98) 2014, 261-313.

   \bibitem{KM}J. Kahn and V. Markovic, {\it  The good pants homology and the  Ehrenpreis Conjecture,} Ann. of Math., (182) 2015, 1-72.
\bibitem{KI}I. Kra, {\it Canonical mappings between Teichm\"{u}ller spaces}, Bull. Amer. Math. Soc. (4) 1981, 143-179.
\bibitem{Odd}C. Odden, {\it The Baseleaf Preserving Mapping Class Group of the Universal Hyperbolic Solenoid}, Tran. Amer. Math. Soc., (357) 2005, 1829-1858.
\bibitem{PS}R. C. Penner and D. $\check{S}$ari$\acute{c}$. {\it Teichm\"uller theory of the punctured solenoid.}  Geom. Dedicata, (132) 2008, 179-212.
\bibitem{Sar1}D. $\check{S}$ari$\acute{c}$ {\it Earthquakes and Thurston Boundary for the Teichm\"uller Space of the Universal Hyperbolic Solenoid,} Pac. J. Math., (233) 2007, 205-228.
\bibitem{Sar2}D. $\check{S}$ari$\acute{c}$ {\it On Quasiconformal Deformations of the Universal Hyperbolic Solenoid,} J. Anal Math., (105) 2008, 303-343.
\bibitem{Sar3}D. $\check{S}$ari$\acute{c}$ {\it The Teichm\"uller theory of the solenoid,} In Handbook of Teichm\"uller theory. Vol. II,  2009, 811-857.

\bibitem{Sha}I. R. Shafarevich, {\it On some infinite-dimensional groups II,} Math USSR Izvest., (18) 1982, 185-194.

\end{thebibliography}
\end{document}